\documentclass[article,12pt]{amsart}

\usepackage{mathptmx}
\usepackage{amsmath,amsfonts,amssymb}
\usepackage{stmaryrd}
\usepackage{latexsym}

\usepackage{graphicx}
\usepackage{color}
\usepackage{epsfig}

\numberwithin{equation}{section}
\theoremstyle{plain}
\newtheorem{thm}{Theorem}[section]
\newtheorem{rem}{Remark}[section]

\newtheorem{cor}{Corollary}[section]
\newtheorem{lem}{Lemma}[section]

\newcommand{\dF}{\mathbb{F}}
\newcommand{\dR}{\mathbb{R}}
\newcommand{\dE}{\mathbb{E}}
\newcommand{\dP}{\mathbb{P}}
\newcommand{\dZ}{\mathbb{Z}}

\newcommand{\cN}{\mathcal{N}}
\newcommand{\cX}{\mathcal{X}}
\newcommand{\cY}{\mathcal{Y}}
\newcommand{\cR}{\mathcal{R}}
\newcommand{\cS}{\mathcal{S}}

\newcommand{\cC}{\mathcal{C}}

\newcommand{\cH}{\mathcal{H}}

\newcommand{\cF}{\mathcal{F}}

\newcommand{\cM}{<\!\!M\!\!>}

\renewcommand{\leq}{\leqslant}
\renewcommand{\geq}{\geqslant}
\newcommand{\rI}{\mathrm{I}}
\newcommand{\ind}{\mbox{1}\kern-.25em \mbox{I}}
\font\calcal=cmsy10 scaled\magstep1
\def\build#1_#2^#3{\mathrel{\mathop{\kern 0pt#1}\limits_{#2}^{#3}}}
\def\liml{\build{\longrightarrow}_{}^{{\mbox{\calcal L}}}}

\def\videbox{\mathbin{\vbox{\hrule\hbox{\vrule height1ex \kern.5em
\vrule height1ex}\hrule}}}

\title[Concentration inequalities for self-normalized
martingales]{New insights on  concentration inequalities for self-normalized martingales}

\author{Bernard Bercu}
\address{
Universit\'e de Bordeaux, 
Institut de Math\'ematiques de Bordeaux, 
UMR 5251, 351 cours de la lib\'eration, 
33405 Talence cedex, France.}
\email{bernard.bercu@u-bordeaux.fr}
\author{Taieb Touati}
\address{Sorbonne Universit\'e, Laboratoire de Probabilit\'es Statistique et Mod\'elisation, UMR 8001,
Campus Pierre et Marie Curie,
4 Place Jussieu, 75005 Paris cedex, France}
\email{taieb.touati@upmc.fr}


\begin{document}
\begin{abstract}
We propose new concentration inequalities for self-normalized martingales. The main idea is to introduce
a suitable weighted sum of the predictable quadratic variation and the total quadratic variation
of the martingale. It offers much more flexibility and allows us to improve previous concentration inequalities.
Statistical applications on autoregressive process, internal diffusion-limited aggregation process,
and online statistical learning are also provided.
\end{abstract}
\maketitle

\vspace{-2ex}

\section{Introduction}

Let $(M_n)$ be a locally square integrable real martingale adapted to a 
filtration $\dF=(\cF_n)$ with $M_0=0$. The predictable 
quadratic variation and the total quadratic variation of
$(M_n)$ are respectively given by
$$
\cM_n=\sum_{k=1}^n\dE[\Delta M_k^{2}|\mathcal{F}_{k-1}]
\hspace{1cm}\mbox{and}\hspace{1cm}
[M]_n=\sum_{k=1}^n\Delta M_k^{2}
$$
where $\Delta M_n=M_n -M_{n-1}$ with $\cM_0=0$ and $[M]_0=0$. Since the pionner work of
Azuma and Hoeffding \cite{Azuma1967}, \cite{Hoeffding1963}, a wide literature is available on  
concentration inequalities for martingales. We refer the reader to the recent books 
\cite{BDR2015}, \cite{BLM2013},  \cite{DLPKlassLai2007} where the celebrated Azuma-Hoeffding,  Freedman, Bernstein, and De la Pe{\~n}a 
inequalities are provided. Over the last two decades, there has been a renewed interest in this area of probability.
To be more precise, extensive studies have been made in order to establish
concentration inequalities for $(M_n)$ without boundedness assumptions
on its increments \cite{BercuTouati2008}, \cite{Delyon2009}, \cite{DVZ2001}, \cite{FGL2015}, \cite{Pinelis2014}, \cite{Rio2013b}. 
For example, it was established in \cite{BercuTouati2008} that for any positive $x$ and $y$,
\begin{equation}
\label{BT2008}
\dP(|M_n | \geq x,[M]_n\,+\cM_n \leq y)\leq 2 \exp\Bigl(-\frac{x^2}{2y}\Bigr).
\end{equation}
We shall improve inequality \eqref{BT2008} by showing that a special case of our inequalities leads, for any positive $x$ and $y$, to
\begin{equation}
\label{NCIMART8}
\dP(|M_n | \geq x,[M]_n\,+\cM_n \leq y)\leq 2 \exp\Bigl(-\frac{8x^2}{9y}\Bigr).
\end{equation}
Moreover, it was proven by Delyon  \cite{Delyon2009} that for any positive $x$ and $y$,
\begin{equation}
\label{D2009}
\dP(|M_n | \geq x,[M]_n+ 2\cM_n \leq y)\leq 2 \exp\Bigl(-\frac{3x^2}{2y}\Bigr).
\end{equation}
We will show that inequality \eqref{D2009} is a special case of a more general result involving a suitable
weighted sum of $[M]_n$ and $\cM_n$. Furthermore, it was shown by De la Pe{\~n}a and Pang \cite{DLPPang2009}
that for any positive $x$,
\begin{equation}
\label{DLPP2009}
\dP\Bigl(\frac{|M_n|}{\sqrt{[M]_n+ \cM_n+\dE[M_n^2]}}\geq x\sqrt{\frac{3}{2}}\Bigr)\leq  
\Bigl(\frac{2}{3} \Bigr)^{1/3} x^{-2/3} \exp\Bigl(-\frac{x^2}{2}\Bigr).
\end{equation}
We shall improve inequality \eqref{DLPP2009} by using of the tailor-made normalization
\begin{equation}
\label{DEFSNA}
S_n(a)= [M]_n+ c(a)\!\cM_n,
\end{equation}
where for any $a > 1/8$,
\begin{equation}
\label{DEFCA}
c(a) = \frac{2(1 - 2a +2\sqrt{a(a+1)})}{8a-1}.
\end{equation}
The novelty of our approach is that $S_n(a)$ is a suitable weighted sum $\cM_n$ and $[M]_n$. 
In the special case where $\cM_n=[M]_n$, $S_n(a)$ reduces to $S_n(a)=(1+c(a))\cM_n$
and the best choice of $a$ is clearly the one that minimizes $aS_n(a)=a(1+c(a))\cM_n$, that is $a=1/3$.
However, for small values of $n$, the behavior of $\cM_n$ may be totally different from that of $[M]_n$.
Consequently, our approach provides interesting concentration inequalities
in many situations where $\cM_n\neq [M]_n$.
The paper is organised as follows. Section \ref{S-MR} is devoted to our new concentration inequalities
for self-normalized martingales which improve some previous results of Bercu and Touati
\cite{BercuTouati2008}, Delyon \cite{Delyon2009} and De la Pe{\~n}a and Pang  \cite{DLPPang2009}.
Section \ref{S-SA} deals with statistical applications on autoregressive process, internal diffusion-limited aggregation process,
and online statistical learning. All technical proofs are postponed to the Appendices.
\ \vspace{-2ex} \\

\section{Main results}
\label{S-MR}

Our first result holds without any additional assumption on $(M_n)$.

\begin{thm}
\label{T-CIWA}
Let $(M_n)$ be a locally square integrable real martingale.
Then, as soon as $a>1/8$, we have for any positive $x$ and $y$,
\begin{equation}
\label{CIWA1}
\dP
(|M_n | \geq x,  S_n(a)\leq y)\leq 2 \exp\Bigl(-\frac{x^2}{2ay}\Bigr),
\end{equation}
where $S_n(a)= [M]_n+ c(a)\!\cM_n$ and $c(a)$ is given by \eqref{DEFCA}.
\end{thm}

\begin{rem}
The function $c$ is positive, strictly convex and $c(a)\sim 1/2a$ as $a$ tends to infinity.
Special values are given in Table 1.
\vspace{2ex}

\begin{center}
\begin{tabular}{|c||c|c|c|c|c|c|c|c|}
\hline
a & $9/55$ & $4/21$ &  $9/40$ & $25/96$ & $1/3$  & $9/16$ &  $49/72$ & $4/5$  \\
\hline
c(a)  & $10$ & $6$ &  $4$ & $3$ & $2$ & $1$  & $4/5$ & $2/3$  \\
\hline
\end{tabular}\\
\vspace{0.15cm}
Table 1. Special values of the function $c(a)$
\normalsize
\end{center}
\end{rem}

\begin{rem}
On the one hand, $c(a)=1$ if and only if $a=9/16$. 
Replacing the value $a=9/16$ into \eqref{CIWA1} immediately leads to \eqref{NCIMART8} as
$S_n(a)= [M]_n+ \!\cM_n$. On the other hand, $c(a)=2$ if and only if $a=1/3$. Hence,
in this special case, $S_n(a)= [M]_n+ 2\cM_n$ and we find again \eqref{D2009} by taking 
the value $a=1/3$ into \eqref{CIWA1}.
\end{rem}

\begin{rem}
Some recent results of Johansen and Nielsen \cite{Jo-Nielsen2016b}, \cite{Jo-Nielsen2016a}
on outlier detection algorithms and robust statistical methods for linear regressions
are based on the tail probability for the maximum of a family of martingales. This tail probability
follows from  inequality  \eqref{BT2008}. It should be interesting to improve the results in 
\cite{Jo-Nielsen2016b}, \cite{Jo-Nielsen2016a} by use of our new inequality \eqref{CIWA1},
see also \cite{BN2018}.
\end{rem}

\noindent
Our second result for self-normalized martingales is as follows.

\begin{thm}
\label{T-CISNWA}
Let $(M_n)$ be a locally square integrable real martingale.
Then, as soon as $a>1/8$, we have for any positive $x$ and $y$,
\begin{equation}
\label{CISNWA1}
\dP\left(\frac{|M_n|}{S_n(a)}\geq x, S_n(a)\geq y\right) \leq
2\exp\Bigl(-\frac{x^2y}{2a}\Bigr)
\end{equation}
where $S_n(a)= [M]_n+ c(a)\!\cM_n$ and $c(a)$ is given by \eqref{DEFCA}. Moreover, we also have
for any positive $x$,
\begin{equation} 
\label{CISNWA2}
\dP\Bigl(\frac{|M_n|}{S_n(a)} \geq  x \Bigr) \leq 
2 \inf_{p>1}\left(\dE\Bigl[\exp\Bigl(-\frac{(p-1)x^2S_n(a)}{2a}\Bigr)\Bigr]\right)^{1/p}\!\!\!\!\!.
\end{equation}
\end{thm}

\begin{rem}
In the special case $a=9/16$, we find from \eqref{CISNWA1} and \eqref{CISNWA2}
that for any positive $x$ and $y$,
\begin{equation*}
\dP\left(\frac{|M_n|}{[M]_n+ \cM_n}\geq x, [M]_n+ \cM_n\geq y\right) \leq
2\exp\Bigl(\!-\frac{8x^2y}{9}\Bigr),
\end{equation*}
\begin{equation*} 
\dP\Bigl(\frac{|M_n|}{[M]_n+ \cM_n}\geq x \Bigr)
\leq 2 \inf_{p>1}\left(\dE\Bigl[\exp\Bigl(-\frac{8(p-1)x^2}{9}\Bigl( [M]_n+ \cM_n\Bigr) \Bigr)\Bigr]\right)^{1/p}\!\!\!\!.
\end{equation*}
Similar concentration inequalities for self-normalized martingales can be obtained for $a=1/3$ as well as for other
values of $a>1/8$. In addition, via the same lines as in the proof of Theorem \ref{T-CISNWA}, it is easy to see
that for any positive $x$ and $y$,
\begin{equation}
\label{CISNWA3}
\dP\left(\frac{|M_n|}{\cM_n}\geq x, c(a)\!\cM_n\geq [M]_n+y\right) \leq
2\exp\Bigl(\!-\frac{x^2y}{2ac^2(a)}\Bigr),
\end{equation}
\begin{equation} 
\label{CISNWA4}
\dP\Bigl(\frac{|M_n|}{\cM_n}\geq x, [M]_n \leq  c(a) y\!\cM_n \Bigr)\!
\!\leq\! 2 \inf_{p>1}\!\left(\!\dE\Bigl[\exp\Bigl(-\frac{(p-1)x^2\!\cM_n}{2ac(a)(1+y)}\Bigr)\Bigr]\!\right)^{1/p}\!\!\!\!\!\!\!.
\end{equation}
\end{rem}

\noindent
Our third result deals with missing factors in exponential inequalities
for self-normalized martingales with upper bounds independent of $[M]_n$ or $\cM_n$.

\begin{thm}
\label{T-CISNMISS}
Let $(M_n)$ be a locally square integrable real martingale. Assume that $\dE[|M_n|^p] \!<\! \infty$ for some
$p \geq 2$. Then, as soon as $a\!>\!1/8$, we have for any positive $x$,
\begin{equation}
\label{CISNMISS1}
\dP\Bigl(\frac{|M_n|}{\sqrt{aS_n(a)+(\dE[|M_n|^p])^{2/p}}}\geq \frac{x}{\sqrt{B_q}}\Bigr)\leq  
C_q x^{-B_q} \exp\Bigl(-\frac{x^2}{2}\Bigr)
\end{equation}
where $q=p/(p-1)$ is the H\"older conjugate exponent of $p$,
$$
B_q= \frac{q}{2q-1}
\hspace{1cm}\text{and}\hspace{1cm}
C_q=\Bigl(  \frac{q}{2q-1} \Bigr)^{B_q/2}.
$$
In particular, for $p=2$, we have for any positive $x$,
\begin{equation}
\label{CISNMISS2}
\dP\Bigl(\frac{|M_n|}{\sqrt{aS_n(a)+\dE[M_n^2]}}\geq x\sqrt{\frac{3}{2}}\Bigr)\leq  
\Bigl(\frac{2}{3} \Bigr)^{1/3} x^{-2/3} \exp\Bigl(-\frac{x^2}{2}\Bigr).
\end{equation}
\end{thm}

\begin{rem}
In the special case $a=9/16$, we saw that $S_n(a)=[M]_n+ \cM_n$. Hence,
we deduce from \eqref{CISNMISS2} that for any positive $x$, 
\begin{equation*}
\dP\Bigl(\frac{|M_n|}{\sqrt{a([M]_n+ \cM_n)+\dE[M_n^2]}}\geq x\sqrt{\frac{3}{2}}\Bigr)\leq  
\Bigl(\frac{2}{3} \Bigr)^{1/3} x^{-2/3} \exp\Bigl(-\frac{x^2}{2}\Bigr).
\end{equation*}
Since $a<1$, this inequality clearly leads to
\begin{equation*}
\dP\Bigl(\frac{|M_n|}{\sqrt{[M]_n+ \cM_n+\dE[M_n^2]}}\geq x\sqrt{\frac{3}{2}}\Bigr)\leq  
\Bigl(\frac{2}{3} \Bigr)^{1/3} x^{-2/3} \exp\Bigl(-\frac{x^2}{2}\Bigr).
\end{equation*}
Consequently, in the special case $a=9/16$, \eqref{CISNMISS2}
provides a tighter upper bound than inequality (3.2) in \cite{DLPPang2009}.
Moreover, in the special case  $a=1/3$, we also saw that $S_n(a)= [M]_n+ 2\cM_n$.
Hence, we obtain from \eqref{CISNMISS2} that for any positive $x$, 
\begin{equation*}
\dP\Bigl(\frac{|M_n|}{\sqrt{[M]_n+2\cM_n+3\dE[M_n^2]}}\geq \frac{x}{\sqrt{2}}\Bigr)\leq  
\Bigl(\frac{2}{3} \Bigr)^{1/3} x^{-2/3} \exp\Bigl(-\frac{x^2}{2}\Bigr)
\end{equation*}
which is exactly inequality (3.64) in \cite{BDR2015}.
\end{rem}

\begin{proof}
The proofs are given in Appendices A and B.
\end{proof}

\section{Statistical applications}
\label{S-SA}


\subsection{Autoregressive process}


Consider the first-order autoregressive process given,
for all $n\geq 1$, by
\begin{equation}
\label{AR}
X_{n}=\theta X_{n-1} + \varepsilon_{n}
\end{equation}
where $X_n$ and $\varepsilon_n$ are the observation and the driven noise of the process, respectively. Assume that $(\varepsilon_n)$ is a sequence 
of independent and identically distributed random variables sharing the same
$\cN(0,\sigma^2)$ distribution where $\sigma^2>0$. The process is said to be stable if $| \theta|< 1$, unstable if $| \theta|= 1$, and explosive if $| \theta|> 1$.
We estimate the unknown parameter $\theta$ by the standard least-squares  estimator given, for all $n\geq 1$, by
\begin{equation}
\label{LSAR}
\widehat{\theta}_n=\frac{\sum_{k=1}^nX_{k-1}X_k}{\sum_{k=1}^nX_{k-1}^2}.
\end{equation}
It is well-known that whatever the value of $\theta$ is, $\widehat{\theta}_n$ converges almost surely to $\theta$. Moreover, 
White \cite{White1958} has shown that in the stable case $| \theta|< 1$,
$$
 \sqrt{n} \bigl( \widehat{\theta}_n - \theta \bigr) \liml \cN (0, 1-\theta^2),
$$
while in the explosive case $| \theta|> 1$ with initial value $X_0=0$,
$$
|\theta|^n  \bigl( \widehat{\theta}_n - \theta \bigr) \liml (\theta^2 -1) \cC
$$
where $\cC$ stands for the Cauchy distribution. Furthermore, in the stable case $| \theta|< 1$,
it was proven in \cite{BercuGamboaRouault1997} that the sequence $(\widehat{\theta}_n )$ satisfies a large deviation principle with a convex-concave rate function,
see also \cite{Bercu2001}.  A fairly simple concentration inequality  
for the estimator $\widehat{\theta}_n$ was established in \cite{BercuTouati2008}, whatever the value of $\theta$ is. More precisely,
assume that $X_0$ is independent of $(\varepsilon_n)$ with $\cN(0,\tau^2)$ distribution where $\tau^2 \geq \sigma^2$. Then,
for all $n\geq 1$ and for any positive $x$,  we have
\begin{equation}
\label{APPLARLS}
\dP(|\widehat{\theta}_n-\theta| \geq x)
\leq 
2\exp\Bigl(-\frac{nx^2}{2(1+y_x)}\Bigr)
\end{equation}
where $y_x$ is the unique positive solution of 
the equation $h(y_x)=x^2$ and $h$ is the function given, for any positive $x$, by 
$h(x)=(1+x)\log(1+x) -x$. It follows from \eqref{APPLARLS} that, as soon as
$0<x<1/2$,
$$
\dP\bigl(|\widehat{\theta}_n-\theta| \geq x\bigr)
\leq 
2\exp\Bigl(-\frac{nx^2}{2(1+2x)}\Bigr).
$$
The situation in which $(\varepsilon_n)$ is not normally distributed, is much more difficult to handle. If $(\varepsilon_n)$
is a sequence of independent and identically distributed random variables, uniformly bounded with symmetric distribution, we can
use De la Pe$\tilde{\mbox{n}}$a's inequality \cite{DLP1999} for self-normalized conditionally symmetric martingales, 
to prove concentration inequalities for the least-squares estimator, see \cite{BDR2015}. Our motivation is to establish concentration inequalities 
for $\widehat{\theta}_n$ in the situation where the distribution of $(\varepsilon_n)$ is non-symmetric.

\begin{cor} 
\label{C-AR}
Assume that $(\varepsilon_n)$
is a sequence of independent and identically distributed random variables such that, for all $n \geq 1$,
$$
\varepsilon_n= 
 \left \{ \begin{array}{ccc}
   \,\, 2q & \text{with probability} & p,  \vspace{1ex} \\
   \!\! -2p & \text{with probability} & q,
   \end{array} \nonumber \right.
$$
where $p\in ]0,1/2]$ and $q=1-p$. Moreover, assume that $X_0$ is independent of $(\varepsilon_n)$ with
$|X_0| \geq 2p$. Then, for any $a>1/8$ and for any $x$ in the interval $[0, \sqrt{ad(a)}]$,
we have
\begin{equation}
\label{INEGLS}
\dP\bigl(|\widehat{\theta}_n-\theta| \geq x\bigr)
\leq 
2\exp\Bigl(-\frac{np^2x^2}{a d(a)}\Bigr),
\end{equation}
where
\begin{equation}
\label{DEFDA}
d(a)=\frac{4\bigl(q^2+pq c(a)\bigr)^2}{\bigl(p^2+pq c(a)\bigr)}.
\end{equation}
\end{cor}

\begin{rem}
In the symmetric case $p=1/2$, $\cM_n=[M]_n$, $S_n(a)=(1+c(a))\cM_n$
and $d(a)$ reduces to $d(a)=1+c(a)$. Hence, if $a=1/3$, $c(a)=2$ and $d(a)=3$.
Consequently, we deduce from \eqref{INEGLS} that for any $x$ in $[0, 1]$, 
$$
\dP\bigl(|\widehat{\theta}_n-\theta| \geq x\bigr)
\leq 
2\exp\Bigl(-\frac{nx^2}{4}\Bigr).
$$
Moreover, in the nonsymmetric case $p\neq 1/2$, we have almost surely
$$\frac{p}{q}\cM_n \leqslant [M]_n \leqslant \frac{q}{p} \cM_n.$$ 
For example, if $p=1/3$ and $a=9/16$, $c(a)=1$ and $d(a)=16/3$ which implies that $ad(a)=3$
Therefore, we obtain from \eqref{INEGLS} that for any $x$ in $[0,\sqrt{3}]$, 
$$
\dP\bigl(|\widehat{\theta}_n-\theta| \geq x\bigr)
\leq 
2\exp\Bigl(-\frac{nx^2}{27}\Bigr).
$$
\end{rem}

\begin{proof}
It immediately follows from \eqref{AR} together
with \eqref{LSAR} that for all $n \geq 1$,
\begin{equation}
\label{LSMART}
\widehat{\theta}_n-\theta=\sigma^2\frac{M_n}{\cM_n}
\end{equation}
where $\sigma^2=4pq$ and $(M_n)$ is the locally square integrable real martingale given by
\begin{equation*}
M_n=\sum_{k=1}^nX_{k-1}\varepsilon_k,
\hspace{1cm}
\cM_n=\sigma^2\sum_{k=1}^nX_{k-1}^2,
\hspace{1cm}
[M]_n= \sum_{k=1}^nX_{k-1}^2\varepsilon_k^2.
\end{equation*}
We clearly have 
$d_{min}(a)\cM_n \leq S_n(a) \leq d_{max}(a) \cM_n$
where 
$$
d_{min}(a)= \frac{p}{q} +c(a)
\hspace{1cm}
\mbox{and}
\hspace{1cm}
d_{max}(a)= \frac{q}{p} +c(a).
$$
Hereafter, we obtain from \eqref{CISNWA2} that for any $a>1/8$ and for any $x>0$,
\begin{equation} 
\label{CIAR1}
\dP\bigl( |M_n|  \geq  x S_n(a) \bigr) \leq 
2 \left(\dE\Bigl[\exp\Bigl(-\frac{x^2S_n(a)}{2a}\Bigr)\Bigr]\right)^{1/2}
\end{equation}
which implies via \eqref{LSMART} that for any $x>0$,
\begin{equation} 
\label{CIAR2}
\dP \bigl( |\widehat{\theta}_n-\theta | \geq x  \bigr) \leq 
2 \left(\dE\Bigl[\exp\Bigl(-\frac{x^2  \cM_n}{2a  \sigma^2 d(a)}\Bigr)\Bigr]\right)^{1/2}
\end{equation}
where $d(a)$ is given by \eqref{DEFDA}.
It only remains to find a suitable upper-bound for the Laplace transform of $\cM_n$. We have from \eqref{AR} that
$X_n^2=\theta^2 X_{n-1}^2 + 2 \theta X_{n-1} \varepsilon_n + \varepsilon_n^2$. Hence, if $\cF_n=\sigma(X_0, \ldots, X_n)$,
we obtain that for any real $t$ and for all $n \geq 1$,
\begin{equation} 
\label{CIAR3}
\dE[\exp(tX_n^2)|\cF_{n-1}]
=\exp(t \theta^2 X_{n-1}^2)\Lambda_{n-1}(t)
\end{equation}
where
\begin{equation} 
\label{CIAR4}
\Lambda_{n-1}(t)= p \exp\bigl(4tq^2+4\theta t q X_{n-1}\bigr)+q \exp\bigl(4tp^2-4\theta t p X_{n-1}\bigr).
\end{equation}
It follows from Kearns-Saul's inequality \cite{BDR2015}, \cite{KS1998} that for any real $s$,
\begin{equation} 
\label{CIAR5}
p\exp(qs)+ q\exp(-ps) \leq  \exp\Bigl(\frac{\varphi(p)s^2}{4}\Bigr),
\end{equation}
where the function
$\varphi$ is defined by
$\varphi (p) = (q-p)/ \log(q/p)$.
One can observe that $\varphi (p) \in [0, 1/2]$. Then, we deduce from \eqref{CIAR4} and \eqref{CIAR5} 
with $s=4 \theta t X_{n-1}$ that for any
$t\leq 0$,
$
\Lambda_{n-1}(t) \leq \exp\bigl(4tp^2 + 4\varphi(p)t^2\theta^2 X_{n-1}^2 \bigr)
$
leading to
\begin{equation} 
\label{CIAR6}
\dE[\exp(tX_n^2)|\cF_{n-1}] \leq  \exp\bigl(4tp^2 +  t \theta^2 X_{n-1}^2(1+4\varphi(p)t) \bigr).
\end{equation}
As soon as $t \in[-1/2,0]$, we get from \eqref{CIAR6} that
$\dE[\exp(tX_n^2)|\cF_{n-1}] \leq \exp(4tp^2)$.
Consequently, for any $t\in [-1/2\sigma^2,0]$ and for all $n \geq 1$,
$$
\dE[\exp(t\cM_n)] \leq  \dE[\exp(t\cM_{n-1})] \exp(4tp^2 \sigma^2)
$$
which ensures that
\begin{equation} 
\label{CIAR8}
\dE[\exp(t\cM_n)] \leq \exp(4n tp^2\sigma^2)
\end{equation}
as $|X_0| \geq 2p$. Therefore, it follows from  \eqref{CIAR2} and  \eqref{CIAR8} that
for any $x \in [0, \sqrt{ad(a)}]$,
\begin{equation*}
\dP\bigl(|\widehat{\theta}_n-\theta| \geq x\bigr)
\leq 
2\exp\Bigl(-\frac{np^2x^2}{a d(a)}\Bigr)
\end{equation*}
which achieves the proof of Corollary \ref{C-AR}.
\end{proof}


\subsection{Internal diffusion-limited aggregation process}


Our second application deals with the internal diffusion-limited aggregation process.
This aggregation process, first introduced in Mathematics by Diaconis and Fulton \cite{DiaconisFulton1991}, is a cluster growth model in $\dZ^d$
where explorers, starting from the origin at time $0$, are travelling as a simple random walk on $\dZ^d$
until they reach an uninhabited site that is added to the cluster.
In the special case $d=1$, the cluster is an interval $A(n)=[L_n,R_n]$ which, properly normalized, converges
almost surely to $[-1,1]$.
In dimension $d \geq 2$, Lawler, Bramson and Griffeath \cite{LBG1992} have shown that the limit shape of the
cluster is a sphere, see some refinements in \cite{Lawler1995}. We shall restrict our attention
on the one-dimensional internal diffusion-limited aggregation process.
Consider the simple random walk on the integer number line $\dZ$ starting from the origin
at time $0$. At each step, the explorer moves to the right $+1$ or to the left $-1$ with equal probability $1/2$.
Let $(A(n))$ be the sequence of random subsets of $\dZ$, recursively defined as follows:
$A(0)=\{0\}$ and, for all $n\geq 0$,
\begin{equation*}
   A(n+1) = \left \{ \begin{array}{l}
    A(n)\cup \{L_{n}-1\} \vspace{0.5ex} \\
    A(n)\cup \{R_{n}+1\}
   \end{array} \nonumber \right.
\end{equation*}
if the explorer leaves $A(n)$ by the left side or by the right side, respectively, where
$L_n$ and $R_n$ stand for be the minimum and the maximum of $A(n)$. To be more precise,
$$
A(n)=\{L_n,L_n + 1,\ldots,R_n-1, R_n\}.
$$ 
The random set $A(n)$ is characterized by $X_n=L_n+R_n$ as $R_n-L_n=n$,
$$
L_n=\frac{X_n-n}{2} \hspace{1cm}\text{and}\hspace{1cm} R_n=\frac{X_n+n}{2}.
$$
One can observe that $L_n$ and $R_n$ correspond to the number of negative and positive sites of
$A(n)$, respectively. It was proven in \cite{DiaconisFulton1991} that
$$
\lim_{n\rightarrow \infty} \frac{X_n}{n}=0 \hspace{1cm}\text{a.s.}
$$
and
$$
\frac{X_n}{\sqrt{n}} \liml \cN\Bigl(0,\frac{1}{3} \Bigr),
$$
which are respectively equivalent to the almost sure convergences
$$
\lim_{n\rightarrow \infty} \frac{L_n}{n}=-\frac{1}{2} \hspace{1cm}\text{and} \hspace{1cm}
\lim_{n\rightarrow \infty} \frac{R_n}{n}=\frac{1}{2} 
$$
and the asymptotic normalities
$$
\frac{1}{\sqrt{n}}\Bigl(L_n + \frac{n}{2} \Bigr) \liml \cN\Bigl(0,\frac{1}{12} \Bigr)
\hspace{1cm}\text{and} \hspace{1cm}
\frac{1}{\sqrt{n}}\Bigl(R_n - \frac{n}{2} \Bigr) \liml \cN\Bigl(0,\frac{1}{12} \Bigr).
$$
It is possible to prove from Azuma-Hoeffding's inequality \cite{BDR2015} that for any positive $x$,
\begin{equation}
\label{IDLAAH}
\dP\Bigl( \frac{|X_n|}{n} \geq x \Bigr) \leq 2 \exp\Bigl(-\frac{3}{8} n x^2\Bigr).
\end{equation} 
Our goal is to improve this inequality with a suitable use of Theorems \ref{T-CIWA}
and  \ref{T-CISNMISS}.

\begin{cor} 
\label{C-IDLA}
For any $a$ in the interval $]1/8,9/16]$ and for any positive $x$,  we have
\begin{equation}
\label{INEGIDLA1}
\dP\Bigl( \frac{|X_n|}{n} \geq x \Bigr) \leq 2 \exp\Bigl(-\frac{nx^2}{2ac_n(a)} \Bigr)
\end{equation}
and
\begin{equation}
\label{INEGIDLA2}
\dP\Bigl(\frac{|X_n|}{\sqrt{n}}\geq x\Bigr)\leq  
(d_n(a))^{1/3} x^{-2/3} \exp\Bigl(-\frac{x^2}{3d_n(a)}\Bigr)
\end{equation}
where
\begin{equation}
\label{DEFCNA}
c_n(a)=\Bigl(\frac{2n+1}{n+1}\Bigr)\Bigl(\frac{3+c(a)}{6}\Bigr)+\Bigl(\frac{n(1+c(a)) + 2c(a)}{(n+1)^2}\Bigr)
\end{equation}
and
\begin{equation}
\label{DEFDNA}
d_n(a)=c_n(a)+\Bigl(\frac{n+2}{3n}\Bigr).
\end{equation}
\end{cor}

\begin{rem}
The calculation of $c_n(a)$ and $d_n(a)$ looks rather complicated. However, it is absolutely not the case.  As a matter of fact, if $a=1/3$,
$c(a)=2$ and it immediately follows from \eqref{DEFCNA} that
$$
c_n(a)= \frac{10n^2+33n+29}{6(n+1)^2}.
$$
It is not hard to see that for all $n \geq 1$, $c_n(a) \leq 3$ and $d_n(a) \leq 4$.
Consequently, we can deduce from \eqref{INEGIDLA1} that for any positive $x$, 
$$
\dP\Bigl( \frac{|X_n|}{n} \geq x \Bigr) \leq 2 \exp\Bigl(-\frac{nx^2}{2} \Bigr)
$$
which clearly outperforms inequality \eqref{IDLAAH}. In addition, \eqref{INEGIDLA2} implies 
that for any positive $x$, 
$$
\dP\Bigl(\frac{|X_n|}{\sqrt{n}}\geq x\Bigr)\leq  
\Bigl(\frac{2}{x} \Bigr)^{2/3} \exp\Bigl(-\frac{x^2}{12}\Bigr).
$$
Moreover, if $a=25/96$,
$c(a)=3$ and we obtain from \eqref{DEFCNA} that
$$
c_n(a)= \frac{2n^2+5n+7}{(n+1)^2}.
$$
Hence, for all $n \geq 1$, $c_n(a) \leq 7/2$ and $d_n(a) \leq 9/2$. Therefore, we find from \eqref{INEGIDLA1} that for any positive $x$, 
$$
\dP\Bigl( \frac{|X_n|}{n} \geq x \Bigr) \leq 2 \exp\Bigl(-\frac{96 nx^2}{175} \Bigr)
$$
which improves the above inequality for $a=1/3$. Finally, we also deduce from \eqref{INEGIDLA2} that for any positive $x$, 
$$
\dP\Bigl(\frac{|X_n|}{\sqrt{n}}\geq x\Bigr)\leq  
\Bigl(\frac{3}{\sqrt{2}x} \Bigr)^{2/3} \exp\Bigl(-\frac{2x^2}{27}\Bigr).
$$
\end{rem}

\begin{proof}
It follows from a  stopping time argument for gambler's ruin \cite{DiaconisFulton1991}, \cite{Freedman1965}
that for all $n \geq 1$,
\begin{eqnarray}
\label{GR}
\dP(X_{n}=X_{n-1}-1|X_{n-1}) &=& \frac{(n+1+X_{n-1})}{2(n+1)},  \notag \\
\dP(X_{n}=X_{n-1}+1|X_{n-1}) &=&\frac{(n+1-X_{n-1})}{2(n+1)}.
\end{eqnarray} 
We obtain from \eqref{GR} that for all $n \geq 1$, 
$X_n=X_{n-1} + \xi_n$ where
the distribution of the increment $\xi_n$ given $\cF_{n-1}$ is a Rademacher $\cR(p_n)$ distribution with
$$
p_n=\frac{(n+1-X_{n-1})}{2(n+1)}.
$$
Hence, we clearly have
\begin{equation}
\label{CMEAN1X}
\dE[X_n | \cF_{n-1} ]= X_{n-1} + \dE[\xi_n | \cF_{n-1} ] = \Bigl(\frac{n}{n+1}\Bigr) X_{n-1}
\end{equation}
and
\begin{equation}
\label{CMEAN2X}
\dE[X_n^2 | \cF_{n-1} ]= X_{n-1}^2 + 2 X_{n-1}\dE[\xi_n | \cF_{n-1} ] +1= 1+\Bigl(\frac{n-1}{n+1}\Bigr) X_{n-1}^2.
\end{equation}
Let $(M_n)$ be the sequence defined by $M_n=(n+1)X_n$. We immediately deduce from
\eqref{CMEAN1X} and \eqref{CMEAN2X} that $(M_n)$ is a locally square integrable real martingale
such that
$$
\cM_n=\sum_{k=1}^n (k+1)^2- \sum_{k=1}^n X_{k-1}^2.
$$
Moreover, for all $n \geq 1$, $|X_n| \leq n$. Hence,
$$
[M]_n = \sum_{k=1}^n((k+1)X_k -kX_{k-1})^2= \sum_{k=1}^n( k\xi_k +X_{k})^2 \leq 3 \sum_{k=1}^n k^2 + \sum_{k=1}^n X_{k}^2.
$$
One can observe that we always have $\cM_n\neq [M]_n$.
In addition,
\begin{equation}
\label{INEGSNA1}
S_n(a)  \leq  (3+c(a)) \sum_{k=1}^n k^2 +(1-c(a)) \sum_{k=1}^{n-1} X_{k}^2 +X_n^2 +c(a)n(n+2).
\end{equation}
We already saw that for any $a \in ]1/8,9/16]$, $c(a) \geq 1$. Therefore, we obtain from
\eqref{INEGSNA1} that for any $a \in ]1/8,9/16]$,
\begin{equation}
\label{MAJSNA}
S_n(a)  \leq  (3+c(a)) \sum_{k=1}^n k^2 + n(n+c(a)(n+2)) \leq n(n+1)^2 c_n(a)
\end{equation}
where $c_n(a)$ is given by \eqref{DEFCNA}. Hence, it follows from  \eqref{CIWA1} with 
$y\!=\! n(n+1)^2 c_n(a)$ that for any $a \!\in ]1/8,9/16]$ and for any positive $x$,
\begin{eqnarray*}
\dP \Bigl( \frac{|X_n|}{n} \geq x \Bigr) &=& \dP\bigl( |M_n|\geq x n(n+1)\bigr), \\
&=& \dP\bigl( |M_n|\geq x n(n+1), S_n(a) \leq y \bigr), \\
& \leq & 2 \exp\Bigl(-\frac{nx^2}{2ac_n(a)} \Bigr),
\end{eqnarray*}
which is exactly what we wanted to prove. Furthermore, we can deduce from identity 
\eqref{CMEAN2X} that for all $n \geq 1$,
\begin{equation*}
\dE[X_n^2]= 1+\Bigl(\frac{n-1}{n+1}\Bigr) \dE[X_{n-1}^2]
\end{equation*}
where the initial value $\dE[X_1^2]=1$. It implies that for all $n \geq 1$,
\begin{equation}
\label{MEANX}
\dE[X_n^2]= \frac{(n+2)}{3}.
\end{equation}
Consequently, we immediately obtain from \eqref{MEANX} that for all $n \geq 1$,
\begin{equation}
\label{MEANM}
\dE[M_n^2] =\frac{(n+1)^2(n+2)}{3}.
\end{equation}
Finally, we find from \eqref{CISNMISS2} together with \eqref{DEFDNA}, \eqref{MAJSNA} and \eqref{MEANM} that for
any $a \in ]1/8,9/16]$ and for any positive $x$,
\begin{equation*}
\dP\Bigl(\frac{|X_n|}{\sqrt{n d_n(a)}}\geq x\sqrt{\frac{3}{2}}\Bigr)\leq  
\Bigl(\frac{2}{3} \Bigr)^{1/3} x^{-2/3} \exp\Bigl(-\frac{x^2}{2}\Bigr)
\end{equation*}
which clearly leads to \eqref{INEGIDLA2}, completing the proof of Corollary \ref{C-IDLA}.
\end{proof}


\subsection{Online statistical learning}


Our third application is devoted to the study of the statistical risk of hypothesis during an online 
learning process using concentration inequalities for martingales.
We refer the reader to the survey of Cesa-Bianchi and Lugosi \cite{CBLugosi2006} for a rather
exhaustive description of the the underlying theory concerning online learning, as well as
to the recent book of Hazan \cite{Hazan2016} for online convex optimization. 
Our approach is based on the contributions of Cesa-Bianchi {\it et al.} \cite{CBCG2004}, \cite{CBG2008}
dealing with the statistical risk of hypothesis in the situation where 
the ensemble of hypothesis is produced by training a learning algorithm incrementally on a
data set of independent and identically distributed random variables. Their bounds 
rely on Freedman concentration inequality for martingales \cite{Freedman1975}.
Consider the task of predicting a sequence in an online manner with inputs and outputs taking values 
in some abstract measurable spaces $\cX$ and $\cY$, respectively.
We call hypothesis $H$, the classifier or regressor generated by a learning algorithm
after training.  The predictive performance of hypothesis $H$
is evaluated by the theoritical risk denoted $R(H)$, which is the expected loss on a realisation 
$(X,Y) \in \cX \times \cY$ drawn from the underlying distribution
\begin{equation}
\label{TRisk}
R(H)=\dE[\ell(H(X),Y)]
\end{equation}
where $\ell$ is a nonnegative and bounded loss function. For the sake of simplicity, we assume
that $\ell$ is bounded by $1$.
Denote by $\cS_n=\{(X_1,Y_1),\ldots,(X_n,Y_n)\}$ a training data set
of independent random variables sharing the same unknown distribution as $(X,Y)$.
Our goal is to predict $Y_{n+1} \in \cY$ given $X_{n+1} \in \cX$, on the basis of $\cS_n$.
Let $\cH_n=\{H_0,H_1,\ldots,H_{n-1} \}$ be a finite ensemble of hypothesis 
generated by an online learning algorithm where the initial hypothesis $H_0$ is
arbitrarily chosen.The empirical risk associated with the ensemble of hypothesis
$\cH_n$ and the training data set $\cS_n$ is given by
\begin{equation}
\label{ERisk}
\widehat{R}_n=\frac{1}{n}\sum_{k=1}^{n}\ell(H_{k-1}(X_k),Y_k).
\end{equation}
Denote by $R_n$ the average risk associated with the ensemble of hypothesis $\cH_n$,
\begin{equation}
\label{AriskH}
R_n=\frac{1}{n}\sum_{k=1}^{n} R(H_{k-1}).
\end{equation}
Our bound on the average risk $R_n$ is as follows.

\begin{cor}
\label{C-OSL}
Let $\cH_{n}=\{H_0,H_1,...,H_{n-1}\}$ be a finite ensemble of hypothesis
generated by a learning algorithm. Then, for any $a$ in the interval $]1/8,9/16]$
and for any positive $x$, we have
\begin{equation}
\label{OSLR1}
\dP( R_n \geq \widehat{R}_n +x ) \leq 
\exp\Bigl(-\frac{nx^2}{2a(1+c(a)V_n)}\Bigr),
\end{equation}
where
\begin{equation}
\label{VER}
V_n=\frac{1}{n}\sum_{k=1}^{n}\dE[\ell^2(H_{k-1}(X),Y)].
\end{equation}
In other words, for any $0< \delta \leqslant 1$,
\begin{equation}
\label{OSLR2}
\dP\Bigl( R_n \geq \widehat{R}_n + \sqrt{\frac{d_n(a)}{n}}\Bigr) \leq \delta,
\end{equation}
where
$
d_n(a)=-2a(1+c(a)V_n)\log (\delta).
$
Moreover, denote $m(a)=\max(4(1+c(a)),c^2(a))/2$. Then, 
for any $0< \delta \leqslant 1$ and
for all $n \geq -am(a) \log (\delta)$, we also have
\begin{equation}
\label{OSLR3}
\dP\Bigl( R_n \geq  \widehat{R}_n -ac(a) \frac{\log(\delta)}{n}+ \frac{1}{2} \sqrt{\Delta_n(a)} \Bigr)
\leq \delta
\end{equation}
where $\Delta_n(a)=B_n(a)\bigl(4+4c(a)\widehat{R}_n +c^2(a)B_n(a)\bigr)$ with
$B_n(a)= -2a\log (\delta)/n$. 
\end{cor}

\begin{rem}
On the one hand, \eqref{OSLR2} improves the deviation inequality
given in Proposition 1 of Cesa-Bianchi, Conconi and Gentile \cite{CBCG2004}, 
$$
\dP\Bigl( R_n \geq \widehat{R}_n + \sqrt{\frac{-2\log(\delta)}{n}}\Bigr) \leq \delta,
$$
as $V_n$ is always smaller than $1$. 
On the other hand, \eqref{OSLR3} is drastically more accurate than the deviation inequality
given in Proposition 2 of Cesa-Bianchi and Gentile \cite{CBG2008},
\begin{equation}
\label{CBG}
\dP\Bigl( R_n \geq \widehat{R}_n + \frac{36}{n}\log\Bigl( \frac{n  \widehat{R}_n +3}{\delta}\Bigr)
+ 2\sqrt{\frac{\widehat{R}_n}{n}\log\Bigl( \frac{n  \widehat{R}_n +3}{\delta}\Bigr)}\Bigr) \leq \delta.
\end{equation} 
Indeed, one can observe that the right-hand sides of \eqref{OSLR3} and \eqref{CBG} are increasing functions
of $\widehat{R}_n$.  The smallest value in \eqref{CBG} for $\widehat{R}_n=0$ is given by
$36\log(3/\delta)/n$. Consequently, inequality  \eqref{CBG} is only effective for $n \geq 36\log(3/\delta)$, which implies
that $n$ must always be greater than $40$. For example, if $\delta=1/5$, it is necessary to assume that $n \geq 36\log(15)$, that is $n\geq 98$.
If $a=1/3$, then $c(a)=2$ and $m(a)=6$. Consequently, inequality \eqref{OSLR3} is interesting as soon as $n \geq -2 \log(\delta)$.
For example, if $\delta=1/5$, it is necessary to assume that $n\geq 4$. For instance, if $\delta=1/5$, $n=100$ and $a=1/3$, the smallest values
in \eqref{OSLR3} and \eqref{CBG} are respectively given by $0.220$ and $0.975$. Finally, for all values of $\delta$, $n$ and $a$, 
one can easily check that \eqref{OSLR3} is always sharper than \eqref{CBG}.
\end{rem}
 
\begin{proof}
Let $(M_n)$ be the locally square integrable real martingale given by
\begin{equation}
\label{MartRisk}
M_n=\sum_{k=1}^{n}\bigl(R(H_{k-1})-\ell(H_{k-1}(X_{k}),Y_{k})\bigr),
\end{equation}
where we recall that $R(H)=\dE[\ell(H(X),Y)]$. We clearly have
$$
\cM_n = \sum_{k=1}^{n}\bigl(\dE[\ell^2(H_{k-1}(X),Y))]-R^2(H_{k-1})\bigr),
$$
$$
[M]_{n} = \sum_{k=1}^{n}\bigl( R(H_{k-1})-\ell(H_{k-1}(X_{k}),Y_{k})\bigr)^2.
$$
Consequently, for any $a \in ]1/8,9/16]$, 
$$
S_n(a) \leq (1-c(a)) \sum_{k=1}^{n}R^2(H_{k-1}) 
+ \sum_{k=1}^n  \ell^2(H_{k-1}(X_{k}),Y_{k}) + c(a)  \sum_{k=1}^n \dE[\ell^2(H_{k-1}(X),Y)]
$$
Hence, as $c(a) \geq 1$ and the loss function $\ell$ is bounded by $1$, we obtain from \eqref{VER} that
$S_n(a) \leq n(1+c(a)V_n)$. Therefore,
it follows from \eqref{CIWA1} with 
$y\!=\! n(1+c(a)V_n)$ that for any $a \in ]1/8,9/16]$ and for any positive $x$,
\begin{equation}
\label{CRisk1}
\dP(M_n \geq x) \leq \exp\Bigl(-\frac{x^2}{2na(1+c(a)V_n)}\Bigr)
\end{equation}
which immediately leads to
\begin{equation}
\label{CRisk2}
\dP\Bigl(\frac{M_n}{n} \geq x\Bigr) \leq \exp\Bigl(-\frac{n x^2}{2a(1+c(a)V_n)}\Bigr).
    \end{equation}
However, we clearly have from \eqref{MartRisk} that
$$
\frac{1}{n}M_n = R_n - \widehat{R}_n.
$$
Hence, \eqref{CRisk2} immediately  implies \eqref{OSLR1} and \eqref{OSLR2}. It only remains to prove \eqref{OSLR3}.
Since the loss function $\ell$ is bounded by $1$, we obtain from \eqref{VER} that $V_n \leq R_n$. Consequently, 
if $B_n(a)= -2a\log (\delta)/n$, 
\eqref{OSLR2}
ensures that for any $0< \delta \leqslant 1$,
\begin{equation}
\label{CRisk3}
\dP\Bigl( R_n \geq \widehat{R}_n + \sqrt{B_n(a)(1+c(a)R_n)}\Bigr) \leq \delta.
\end{equation}
Therefore, we deduce from \eqref{CRisk3} that for any $0< \delta \leqslant 1$,
\begin{equation}
\label{CRisk4}
\dP\bigl( \Phi_a(R_n)  \geq  \widehat{R}_n  \bigr) \leq \delta
\end{equation}
where the function $\Phi_a$ is defined, for all $x$ in $[0,1]$, by
$\Phi_a(x)=x-\sqrt{B_n(a)(1+c(a)x)}$. It is not hard to see that, as soon as $n \geq -am(a) \log (\delta)$ with $m(a)=\max(4(1+c(a)),c^2(a))/2$,
$\Phi_a$ is a strictly convex and increasing function on $[0,1]$. Then, $\Phi_a$ is invertible and it follows from straightforward calculations that
$$
\Phi_a^{-1}(x)=x  + \frac{1}{2} c(a) B_n(a)+ \frac{1}{2} \sqrt{B_n(a)\bigl(4+4c(a)x +c^2(a)B_n(a)\bigr)}.
$$
Finally, we immediately obtain from \eqref{CRisk4} that
\begin{equation}
\label{CRisk5}
\dP\bigl( \Phi_a(R_n)  \geq  \widehat{R}_n  \bigr)=\dP\bigl( R_n  \geq   \Phi_a^{-1}(\widehat{R}_n)  \bigr) \leq \delta
\end{equation}
which is exactly inequality \eqref{OSLR3}, completing the proof of Corollary \ref{C-OSL}.
\end{proof} 

\section*{Appendix A \\ Two keystone lemmas}
\renewcommand{\thesection}{\Alph{section}}
\renewcommand{\theequation}{\thesection.\arabic{equation}}
\setcounter{section}{1}
\setcounter{subsection}{0}
\setcounter{equation}{0}

Our first lemma deals with a sharp upper bound on the Hermite generating function associated with
a centered random variable $X$.
\noindent
\begin{lem} 
\label{L-HERMITE}
Let $X$ be a square integrable random variable with mean zero and variance
$\sigma^2$. For all $t\in \dR$, denote
\begin{equation}
\label{DEFL}
L(t)=\dE\Bigl[\exp\Bigl(tX-\frac{at^2}{2}X^2\Bigr)\Bigr]
\end{equation}
with $a>1/8$.
Then, for all $t\in \dR$,
\begin{equation}
\label{MAJLD}
L(t)\leq 1+\frac{b(a)t^2}{2}\sigma^2,
\end{equation}
where
\begin{equation}
\label{DEFBA}
b(a) = \frac{2a(1 - 2a +2\sqrt{a(a+1)})}{8a-1}.
\end{equation}
\end{lem}

\begin{proof}
The proof of Lemma \ref{L-HERMITE} relies on the following Hermite inequality, see also Proposition 12 in \cite{Delyon2009}
for the special value $a=1/3$. For all $x\in \dR$, we have
\begin{equation}
\label{HERMITE}
\exp\Bigl(x-\frac{ax^2}{2}\Bigr)\leq 1 + x + \frac{b(a)x^2}{2}.
\end{equation}
As a matter of fact, let
\begin{equation}
\label{VARPHI}
\varphi_a(x)= \log \Bigl( 1 + x + \frac{bx^2}{2} \Bigr) - x + \frac{ax^2}{2},
\end{equation}
where $b=b(a)$. It is of course necessary to assume that $b>1/2$ which ensures that $1+x+bx^2/2$ is positive
whatever the value of $x$ is. We clearly have
\begin{equation}
\label{DVARPHI}
\varphi_a^{\prime}(x) = \Bigl( 1 + x + \frac{bx^2}{2} \Bigr)^{-1}xP_{a,b}(x),
\end{equation}
where the second degree polynomial $P_{a,b}$ is given by
$$
P_{a,b}(x)=   \frac{abx^2}{2} +  \frac{(2a-b)x}{2} + a+b-1.
$$
Hereafter, we assume that $a>1/8$ and $b\neq 1-a$. The only positive root of the discriminant of $P_{a,b}$ is given by $b=b(a)$. As soon as $b\geq b(a)$, 
we have for all $x \in \dR$, $P_{a,b}(x) \geq 0$. Consequently, we deduce from \eqref{DVARPHI} that the function $\varphi_a$
reaches its minimum for $x=0$. Since $\varphi_a^{\prime}(0)=0$ and
$\varphi_a(0)=0$, we find that for all $x \in \dR$, $\varphi_a(x) \geq 0$ which immediately leads to \eqref{HERMITE}.
Therefore, we obtain from
\eqref{HERMITE} that for all $t \in \dR$,
$$
L(t)=\dE\Bigl[\exp\Bigl(tX-\frac{at^2}{2}X^2\Bigr)\Bigr] \leq \dE\Bigl[ 1 + tX + \frac{b(a)t^2X^2}{2} \Bigr]
= 1 + \frac{b(a) t^2}{2}\sigma^2
$$
which is exactly what we wanted to prove.
\end{proof}
\noindent
Our second exponential supermartingale lemma is as follows.
\begin{lem} 
\label{L-SUPERMG}
Let  $(M_n)$ be a locally square integrable real martingale. For all $t\in \dR$ and $n \geq 0$, denote
\begin{equation}
\label{DEFVN}
V_n(t)=\exp\Bigl(tM_n-\frac{at^2}{2}[M]_n - \frac{b(a)t^2}{2} \!\cM_n\Bigr)
\end{equation}
with $a>1/8$.
Then, $(V_n(t))$ is a positive supermartingale such that $\dE[V_n(t)] \leq 1$.
\end{lem}

\begin{proof}
For all $t \in \dR$ and for all $n\geq 1$,  we clearly have
\begin{equation*}
V_n(t)=V_{n-1}(t)\exp\Bigl(t\Delta M_n - \frac{at^2}{2}\Delta [M]_n - \frac{b(a)t^2}{2} \Delta \!\cM_n\Bigr),
\end{equation*}
where $\Delta M_n=M_n-M_{n-1}$, $\Delta [M]_n=\Delta M_n^2$ and 
$\Delta \!\cM_n=\dE[\Delta M_n^2|\cF_{n-1}]$. Hence, we deduce from
Lemma \ref{L-HERMITE} that for all $t\in \dR$,
\begin{eqnarray*}
\dE[V_n(t)|\cF_{n-1}] &\leq & V_{n-1}(t)\exp\Bigl( - \frac{b(a)t^2}{2}\Delta \!\cM_n\Bigr)\Bigl(1+\frac{b(a)t^2}{2}\Delta \! \cM_n\Bigr),\\
&\leq &  V_{n-1}(t)
\end{eqnarray*}
via the elementary inequality $1+x \leq \exp(x)$. 
Consequently, for all $t\in \dR$, $(V_n(t))$ is a positive supermartingale satisfying for all $n\geq 1$,
$\dE[V_n(t)]\leq \dE[V_{n-1}(t)]$, which implies that $\dE[V_n(t)]\leq \dE[V_0(t)]=1$. 
\end{proof}
\section*{Appendix B \\ Proofs of the main results}
\renewcommand{\thesection}{\Alph{section}}
\renewcommand{\theequation}{\thesection.\arabic{equation}}
\setcounter{section}{2}
\setcounter{subsection}{0}
\setcounter{equation}{0}

\noindent
{\bf Proof of Theorem \ref{T-CIWA}.}
For any positive $x$ and $y$, let
$$A_n=\Bigl\{|M_n|\geq x, aS_n(a) \leq y\Bigr\}$$
with $aS_n(a)= a[M]_n + b(a) \!\cM_n$. We have the decomposition
$A_n=A_n^+\cup A_n^-$ where 
$$
A_n^+ =\Bigl\{M_n\geq x, aS_n(a) \leq y\Bigr\} \hspace{1cm} \text{and} \hspace{1cm}
A_n^-=\Bigl\{M_n\leq \!-x, aS_n(a) \leq y\Bigr\}.
$$
It follows from Markov's inequality together with Lemma \ref{L-SUPERMG} that for all positive $t$,
\begin{eqnarray*}
\dP(A_n^+)&\leq&\dE\Bigl[\exp\Bigl(tM_n - tx\Bigr)\rI_{A_n^+}\Bigr],\\
&\leq &\dE\Bigl[\exp\Bigl(tM_n - \frac{t^2}{2}aS_n(a)\Bigr)\exp\Bigl(\frac{t^2}{2}aS_n(a)- tx \Bigr)\rI_{A_n^+}\Bigr],\\
&\leq &\exp\Bigl(\frac{t^2y}{2}- tx\Bigr) \dE[V_n(t)], \\
&\leq & \exp\Bigl(\frac{t^2y}{2}- tx\Bigr).
\end{eqnarray*}
Hence, by taking the optimal value $t=x/y$ in the above inequality, we find that
\begin{equation*}
\dP(A_n^+)\leq \exp\Bigl(-\frac{x^2}{2y}\Bigr).
\end{equation*}
We also obtain the same upper bound for $\dP(A_n^-)$ which ensures that
\begin{equation}
\label{UBPAN}
\dP(A_n)\leq 2 \exp\Bigl(-\frac{x^2}{2y}\Bigr).
\end{equation}
Finally, inequality \eqref{UBPAN} clearly leads to  \eqref{CIWA1} replacing $y$ by $ay$.
\hfill $\videbox$
\ \vspace{1ex} \\
\noindent
{\bf Proof of Theorem \ref{T-CISNWA}.}
For any positive $x$ and $y$, let 
$$B_n=\Bigl\{|M_n| \geq x S_n(a), S_n(a) \geq y \Bigr\}=B_n^+\cup B_n^-,$$
where 
$$ B_n^+ = \Bigl\{ M_n\geq xS_n(a), S_n(a) \geq y \Bigr \}
\hspace{0.5cm} \text{and} \hspace{0.5cm}
B_n^- = \Bigl\{ M_n \leq \!- xS_n(a), S_n(a) \geq y \Bigr \}.
$$
We have for all positive $t$,
\begin{eqnarray}
\label{UBPBN1}
\dP(B_n^+) &\leq & \dE\Bigl[\exp\Bigl(\frac{t}{2}M_n - \frac{tx}{2}S_n(a) \Bigr)\rI_{B_n^+}\Bigr], \nonumber \\
&\leq &\dE\Bigl[\exp\Bigl(\frac{t}{2}M_n - \frac{t^2}{4}aS_n(a)\Bigr)
\!\exp\Bigl(\frac{t}{4}(ta- 2x)S_n(a)\Bigr)\rI_{B_n^+}\!\Bigr]\!.
\end{eqnarray}
Consequently, we find from \eqref{UBPBN1} with the particular choice $t=x/a$ that
$$
\dP(B_n^+) \leq \exp\Bigl(-\frac{x^2y}{4a}\Bigr) \dE[ \sqrt{V_n(t)} \rI_{B_n^+}].
$$
Hence, we deduce from Cauchy-Schwarz's inequality together with Lemma \ref{L-SUPERMG} that
$$
\dP(B_n^+) \leq \exp\Bigl(-\frac{x^2y}{4a}\Bigr) \sqrt{\dP(B_n^+)}
$$
which leads to
\begin{equation}
\label{UBPBN2}
\dP(B_n^+)\leq \exp\Bigl(-\frac{x^2y}{2a}\Bigr).
\end{equation}
By the same token, we obtain the same upper bound holds for $\dP(B_n^-)$ which clearly implies \eqref{CISNWA1}.
Furthermore, for any positive $x$, let 
$$C_n=\Bigl\{|M_n|\geq x S_n(a) \Bigr\}=C_n^+\cup C_n^-,$$
where
$$
C_n^+ = \Bigl \{M_n\geq x S_n(a) \Bigr\}
\hspace{0.5cm} \text{and} \hspace{0.5cm}
C_n^- = \Bigl \{M_n\leq \!-x S_n(a)  \Bigr\}.
$$
By Holder's inequality, we have for all positive $t$ and $q>1$,
\begin{eqnarray}
\label{MGINEGDC1}
\dP(C_n^+)&\leq&\dE\Bigl[\exp\Bigl(\frac{t}{q}M_n - \frac{tx}{q}S_n(a)\Bigr)\rI_{C_n^+}\Bigr],\nonumber \\
&\leq &\dE\Bigl[\exp\Bigl(\frac{t}{q}M_n - \frac{t^2a}{2q}S_n(a)\Bigr) \exp\Bigl(\frac{t}{2q}(ta- 2x)S_n(a)\Bigr)\rI_{C_n^+}\Bigr],\nonumber \\
&\leq &\dE\Bigl[\bigl(V_n(t)\bigr)^{1/q} \exp\Bigl(\frac{t}{2q}(ta- 2x)S_n(a)\Bigr)\Bigr],\nonumber \\
&\leq &\left(\dE\Bigl[\exp\Bigl(\frac{tp}{2q}(ta- 2x)S_n(a)\Bigr)\Bigr]\right)^{1/p}\!\!.
\end{eqnarray}
Consequently, as $p/q=p-1$, we can deduce from \eqref{MGINEGDC1} with the optimal value $t=x/a$  
that
\begin{equation*}
\dP(C_n^+)\leq \inf_{p>1}\left(\dE\Bigl[\exp\Bigl(-\frac{(p-1)x^2S_n(a)}{2a}\Bigr)\Bigr]\right)^{1/p}\!\!\!\!.
\end{equation*}
We find the same upper bound for $\dP(C_n^-)$, completing the proof of Theorem \ref{T-CISNWA}.
\hfill $\videbox$
\ \vspace{1ex} \\
\noindent
{\bf Proof of Theorem \ref{T-CISNMISS}.} We already saw from Lemma \ref{L-SUPERMG} that for all $t \in \dR$,
$$
\dE\Bigl[ \exp\Bigl(tA_n-\frac{t^2}{2}B_n^2\Bigr)\Bigr] \leq 1
$$
where $A_n=M_n$ and $B_n^2=a[M]_n +b(a) \!\cM_n$.
It means that the pair of random variables $(A_n,B_n)$ safisties the canonical assumption in \cite{DLPPang2009}. 
Consequently, Theorem \ref{T-CISNMISS} immediately follows from Theorem 2.1 in \cite{DLPPang2009}.
\hfill $\videbox$

\end{document}